\documentclass[reqno, 12pt]{amsart}
\usepackage{pifont}
\usepackage{amsfonts}
\usepackage{}
\usepackage{mathbbold}
\usepackage{mathrsfs}
\usepackage{amssymb,url}
\usepackage{cite}
\date{}

\allowdisplaybreaks[4]
\theoremstyle{plain}
\newtheorem{thm}{Theorem}[section]

\newtheorem{lem}[thm]{Lemma}
\newtheorem{pro}[thm]{Proposition}

\theoremstyle{remark}

\theoremstyle{definition}

\usepackage[colorlinks,linkcolor=blue,urlcolor=red,linktocpage]{hyperref}

\makeatletter
    \newcommand{\rmnum}[1]{\romannumeral #1}
    \newcommand{\Rmnum}[1]{\expandafter\@slowromancap\romannumeral #1@}
  \makeatother

\numberwithin{equation}{section}

\begin{document}
\title
{Orlicz valuations}

\author[]{Jin Li$^{1,2}$}
\author[]{Gangsong Leng$^1$}
\address[address1]{Department of Mathematics, Shanghai University, Shanghai 200444, China}
\address[address2]{Institut f\"{u}r Diskrete Mathematik und Geometrie, TU Wien, Wien 1040 Austria}
 \email[Jin Li]{\href{mailto: Jin Li
<lijin2955@gmail.com>}{lijin2955@gmail.com}}
\email[Gangsong Leng]{\href{mailto: Gangsong Leng
<lenggangsong@163.com>}{lenggangsong@163.com}}

\begin{abstract}
In this paper, Orlicz valuations compatible with $SL(n)$ transforms are classified. Unlike their $L_p$ analogs, the identity \text{operator} and the reflection operator are the only $SL(n)$ \text{compatible} Orlicz valuations (up to dilations). It turns out that the Orlicz projection body operator, the Orlicz centroid body operator and the Orlicz difference body operator are not Orlicz valuations. The property that the Orlicz difference body operator is not an Orlicz valuation plays an important role in characterizing the identity operator and the reflection operator.
\end{abstract}

\subjclass[2010]{52A20, 52B45}

\keywords{Orlicz addition, Orlicz valuations, $SL(n)$ contravariance, $SL(n)$ covariance, Cauchy functional equation}

\thanks{}

\maketitle
\section{Introductions}
The Brunn-Minkowski theory, which merges two elementary notions for sets in Euclidean space, vector addition and volume, is the core of convex geometry. For a comprehensive introduction to the Brunn-Minkowski theory, see Schneider \cite{Sch14} and Gardner \cite{Gar02}. During the last few decades, the $L_p$ analog, the $L_p$ Brunn-Minkowski theory, was developed  by Lutwak, Yang, and Zhang, and many others; see \cite{HS09a,HS09b,Lut93,Lut96,LYZ00b,LYZ00a,LYZ02a,LYZ04a,LYZ05}.

Let $\mathcal {K}_o ^n$ be the set of convex bodies (i.e., compact convex sets in $\mathbb{R}^n$) which contain the origin and $\mathcal {P}_o ^n$ be the set of polytopes in $\mathbb{R}^n$ which contain the origin.

For $1 \leq p \leq \infty$ and arbitrary $K,L \in \mathcal {K}_o ^n$, the \emph{$L_p$ Minkowski sum} of $K$ and $L$ is defined by
\begin{align}\label{205}
h_{K +_p L}(x)^p = h_K (x) ^p + h_L (x)^p
\end{align}
for any $x \in \mathbb{R}^n$, where $h_K$ denotes the support function of $K$ (defined in Section 2). When $p = \infty$, the definition (\ref{205}) should be interpreted as $h_{K +_p L}(x) = \max \{h_K (x) , h_L (x)\}$. When $p=1$, the definition (\ref{205}) gives the ordinary Minkowski addition, and $K,L$ need not contain the origin.

An \emph{$L_p$ Minkowski valuation} is a function $Z : \mathcal{P}_o ^n \to \langle \mathcal {K}_o ^n, +_p \rangle$ such that
\begin{align*}
    Z(K \cup L) +_p Z(K \cap L) = ZK +_p ZL,
\end{align*}
whenever $K,L,K \cup L,K \cap L \in \mathcal{P}_o ^n$. Here $\langle \mathcal {K}_o ^n,+_p \rangle$ denotes that $\mathcal {K}_o ^n$ is equipped with $L_p$ Minkowski addition. For $1 \leq p < \infty$, the $L_p$ Minkowski valuations were characterized as moment bodies, difference bodies and projection bodies by Ludwig \cite{Lud05} for $GL(n)$ compatible valuations and Haberl \cite{Hab12b}, Parapatits \cite{Par14a,Par14b} for $SL(n)$ compatible valuations.

A map $Z: \mathcal{K}_o ^n \to \mathfrak{P}(\mathbb{R}^n)$, the power set of $\mathbb{R}^n$, is called $SL(n)$ contravariant if
$$Z \psi K = \psi ^{-t} ZK$$
for any $K \in \mathcal{K}_o ^n$ and any $\psi \in SL(n)$. The map $Z$ is called $SL(n)$ covariant if
$$Z \psi K = \psi ZK$$
for any $K \in \mathcal{K}_o ^n$ and any $\psi \in SL(n)$.

Notice that $\{ o \}$ is the only invariant set of $\mathbb{R}^n$ under any $SL(n)$ transforms. Thus if $Z$ is $SL(n)$ contravariant (or covariant), then
\begin{align}\label{37}
Z \{ o \} = \{ o \}.
\end{align}

The classification theorem of Haberl \cite{Hab12b} and Parapatits \cite{Par14a} for $SL(n)$ contravariant $L_p$ Minkowski valuations can be written as
\begin{thm}[Haberl \cite{Hab12b} and Parapatits \cite{Par14a}]\label{Lpcontra}
Let $n \geq 3$. A map $Z : \mathcal{P}_o ^n \to \langle \mathcal {K}_o ^n, + \rangle$ is an $SL(n)$ contravariant Minkowski valuation if and only if there exist constants
$c_1,c_2,c_3 \in \mathbb{R}$ with $c_1 \geq 0$ and $c_1 +c_2 +c_3 \geq 0$ such that
$$ZP = c_1 \Pi P + c_2 \Pi_o P + c_3 \Pi_o (-P)$$
for all $P \in \mathcal{P}_o ^n$.

For $1 < p < \infty$, a map $Z : \mathcal{P}_o ^n \to \langle \mathcal {K}_o ^n, +_p \rangle$ is an $SL(n)$ contravariant $L_p$ Minkowski valuation if and only if there exist constants
$c_1,c_2 \geq 0$ such that
$$ZP = c_1 \hat{\Pi}_p^+ P +_p c_2 \hat{\Pi}_p ^- P$$
for all $P \in \mathcal{P}_o ^n$.
\end{thm}
Here $\Pi$ is the projection body operator, $\Pi_o$, $\hat{\Pi}_p^+$ and $\hat{\Pi}_p^-$ are the generalizations of the $L_p$ projection body operator; see Section 2.

The classification theorem of Haberl \cite{Hab12b} and Parapatits \cite{Par14b} for $SL(n)$ covariant $L_p$ Minkowski valuations can be written as
\begin{thm}[Haberl \cite{Hab12b} and Parapatits \cite{Par14b}]\label{lpco}
Let $n \geq 3$ and $1 \leq p < \infty$. A map $Z : \mathcal{P}_o ^n \to \langle \mathcal {K}_o ^n, +_p \rangle$ is an $SL(n)$ covariant $L_p$ Minkowski valuation which is continuous at the line segment $[o,e_1]$ if and only if there exist constants
$c_1,\cdots,c_4 \geq 0$ such that
$$ZP = c_1 P +_p c_2 (-P) +_p c_3 M_p^+ P +_p c_4 M_p ^- P$$
for all $P \in \mathcal{P}_o ^n$.
\end{thm}
Here $M_p^+$, $M_p^-$ are the asymmetric $L_p$ moment body operators; see Section 2.

Initiated by Dehn's solution to Hilbert's third question, valuation theory was first systematically investigated by Hadwiger. His fundamental classification theorem, which characterizes the linear combinations of the intrinsic volumes as the continuous, rigid motion invariant (real-valued) valuations, has many beautiful applications in integral geometry and geometric probability; see Klain and Rota's book \cite{KR97}. For recent variants of Hadwiger's theorem, see \cite{Ale99,Ale01,Ber09a}. Even before \text{Hadwiger}, Blaschke studied $SL(3)$ invariant valuations in $\mathbb{R}^3$. For recent results on $SL(n)$ invariant valuations, see \cite{LR99,Lud02a,LR10,HP14b}.

The first result on convex bodies valued valuations was obtained by Schneider \cite{Sch74} in the 1970s. During the last few decades, after a series of papers by Ludwig \cite{Lud02b,Lud03,Lud05,Lud06}, convex bodies valued valuations were studied quickly, see \cite{Hab2008,Hab09,Hab12a,Hab12b,HL06,HS09a,HS09b,LYL2015,Lud10b,ober2014,Par14a,Par14b,PS12,SS06,Sch08,Sch2010,SW12,Tsa12} and also Ludwig's survey \cite{Lud2006}.

The Orlicz Brunn-Minkowski theory introduced by Lutwak, Yang, and Zhang \cite{LYZ10b}, \cite{LYZ10a} gained momentum after Gardner, Hug, and Weil \cite{GHW14} introduced an appropriate Orlicz addition (the following definition (\ref{3})) using the Orlicz norm and established the Orlicz Brunn-Minkowski inequality. A little \emph{weaker} but useful definition for Orlicz addition was also provided by Xi, Jin and Leng \cite{XJL14} and \text{independently} in \cite{GHW14}; see following definition (\ref{11}). For the Dual Orlicz-Brunn-Minkowski Theory, see \cite{GHWY,JYL2014,ZZX2014}. For the Orlicz Minkowski problem, see \cite{HLYZ10,HuaH12}. For other aspects of the Orlicz Brunn-Minkowski theory, see \cite{Bor2013,CYZ2011,GLD2015,LL2011,Lud10a,LR10,WLH2012,Zhu2012,ZouX2014}

Let $\Phi_2$ be the set of convex functions $\varphi: [0,\infty)^2 \to [0,\infty)$ that are increasing in each variable and satisfy $\varphi(0,0) = 0$ and $\varphi(1,0) = \varphi(0,1)= 1$. For arbitrary $K,L \in \mathcal {K}_o ^n$, $\varphi \in \Phi_2$, the \emph{Orlicz sum} of $K$ and $L$ is defined in \cite{GHW14} by
\begin{align}\label{3}
h_{K +_\varphi L}(x) = \inf \{ \lambda > 0 :\varphi \left(\frac{h_K(x)}{\lambda},\frac{h_L(x)}{\lambda} \right) \leq 1 \}
\end{align}
for any $x \in \mathbb{R}^n$. When both $h_K(x),h_L(x) =0$, $h_{K +_\varphi L}(x)$ should be interpreted as $0$. Note that (\ref{3}) is equivalent to
\begin{align}\label{3a}
\varphi \left(\frac{h_K(x)}{h_{K +_\varphi L}(x)},\frac{h_L(x)}{h_{K +_\varphi L}(x)} \right) =1.
\end{align}
Especially, for $1\leq p < \infty$, if $\varphi (x_1,x_2) = (x_1 ^p + x_2 ^p) ^{1/p}$ for any $0 \leq x_1, x_2 \leq 1$, Orlicz addition is $L_p$ Minkowski addition. If $\varphi (x)$ is the maximum coordinate of $x \in [0,1] ^2$, i.e., $\varphi (x) = \max \{ x_1, x_2 \}$ for any  $0 \leq x_1,x_2 \leq 1$, then Orlicz addition is $L_\infty$ Minkowski addition. Orlicz addition is associative if and only if $+_\varphi = +_p$ for some $1 \leq p \leq \infty$; see \cite[Theorem 5.10]{GHW14}.

Orlicz addition is monotonic, continuous, $GL(n)$ covariant, projection covariant and has the identity property. Also the binary operator $* : (\mathcal{K}_s ^n)^2 \to \mathcal{K}^n$ is projection covariant (or equivalently, continuous and $GL(n)$ covariant) if and only if it is an Orlicz addtion for $\varphi \in \Phi_2$ \cite[Theorem 3.3, Theorem 5.2 and Corollary 5.7]{GHW14}. Here $\mathcal{K}_s ^n$ is the set of o-symmetric convex bodies, and $\mathcal{K}^n$ is the set of convex bodies.

Gardner, Hug, and Weil \cite[Section 5]{GHW14} also show that there exists a $2$-dimensional convex body $M$ independent to $K$ and $L$ such that
\begin{align}\label{4}
h_{K +_\varphi L} (\cdot) = h_M (h_K(\cdot),h_L(\cdot)).
\end{align}
If we combine the valuation property with Orlicz addition, then it is natural to assume that Orlicz addition is commutative. So if $+_\varphi$ is not $+_\infty$,
then there exists a $\varphi _0 \in \Phi$, where $\Phi$ is the set of convex functions $\varphi: [0,\infty) \to [0,\infty)$ that are increasing on $[0,\infty)$ and satisfy $\varphi(0) = 0$ and $\varphi(1) = 1$, such that
$+_\varphi = +_{\widetilde{\varphi}}$ and $\widetilde{\varphi}(x_1,x_2) = \varphi_0(x_1)+\varphi_0(x_2)$ for any $x_1,x_2 \geq 0$ (see \cite[Theorem 5.9]{GHW14}).
We will briefly write $\varphi _0$ as $\varphi$. Then we get a \emph{weaker} definition of Orlicz addition from (\ref{3a}),
\begin{align}\label{11}
\varphi \left( \frac{h_K(x)}{h_{K +_\varphi L}(x)} \right) + \varphi \left( \frac{h_L(x)}{h_{K +_\varphi L}(x)} \right) =1
\end{align}
for any $x \in \mathbb{R}^n$. Also, when both $h_K(x),h_L(x) =0$, $h_{K +_\varphi L}(x)$ should be interpreted as $0$.

In the following, Orlicz addition will be defined by (\ref{11}).

An \emph{Orlicz valuation} for a convex function $\varphi \in \Phi$ is a function $Z: \mathcal {P}_o ^n \rightarrow \langle \mathcal {K}_o ^n,+_\varphi \rangle$ such that
\begin{align}\label{16}
    Z(K \cup L) +_\varphi Z(K \cap L) = ZK +_\varphi ZL,
\end{align}
whenever $K,L,K \cup L,K \cap L \in \mathcal {P}_o ^n$. Here $\langle \mathcal {K}_o ^n,+_\varphi \rangle$ denotes $\mathcal {K}_o ^n$ endowed with Orlicz addition defined by (\ref{11}).

Before Orlicz addition was introduced, Lutwak, Yang, and Zhang (\cite{LYZ10a}, \cite{LYZ10b}) introduced the Orlicz projection bodies and the Orlicz centroid bodies (volume-normalized moment bodies) which are Orlicz analogs of the ($L_p$) projection bodies and the ($L_p$) centroid bodies, respectively. The Orlicz projection body operator is $SL(n)$ contravariant and the Orlicz centroid body operator and also the Orlicz moment body operator are $SL(n)$ covariant; see Section 2 for definitions and more details. The ($L_p$) projection body operator and the ($L_p$) moment body operator (not the ($L_p$) centroid body operator) were characterized as ($L_p$) Minkowski valuations in Theorem \ref{Lpcontra} and Theorem \ref{lpco}, respectively.  However, unlike their $L_p$ analogs, it seems that the Orlicz projection body operator $\Pi_\varphi$ and the Orlicz moment body operator $M_\varphi$ are not Orlicz valuations for any convex function $\psi \in \Phi$. So the question is whether they are Orlicz valuations, and, if not, can we modify the definitions of the Orlicz projection operator and the Orlicz moment body operator to make them be such valuations? By classifying the $SL(n)$ compatible Orlicz valuations, we show that the answers to both questions are negative.

\begin{thm}\label{thm1.3}
Let $n \geq 3$, $\varphi \in \Phi$ and $+_\varphi \neq +_p$ for any $p \geq 1$. A map $Z : \mathcal{P}_o ^n \to \langle \mathcal {K}_o ^n, +_\varphi \rangle$ is an $SL(n)$ contravariant Orlicz valuation for $\varphi$ if and only if
\begin{align*}
ZP = \{ o \}
\end{align*}
for all $P \in \mathcal{P}_o ^n$.
\end{thm}

\begin{thm}\label{thm1.4}
Let $n \geq 3$, $\varphi \in \Phi$ and $+_\varphi \neq +_p$ for any $p \geq 1$. A map $Z : \mathcal{P}_o ^n \to \langle \mathcal {K}_o ^n, +_\varphi \rangle$ is an $SL(n)$ covariant Orlicz valuation for $\varphi$ if and only if there exists a constant $a \geq 0$ such that
\begin{align*}
ZP = aP
\end{align*}
for all $P \in \mathcal{P}_o ^n$, or
\begin{align*}
ZP = -aP
\end{align*}
for all $P \in \mathcal{P}_o ^n$.
\end{thm}

Unlike for the $L_p$ analogs (Theorem \ref{lpco}), we do not need to assume continuity in Theorem \ref{thm1.4}. The Orlicz difference body operator is also not an Orlicz valuation. This property plays an important role in characterizing the identity operator and the reflection operator; see Lemma \ref{lem4.3} and Lemma \ref{lem4.5}.

Note that if the condition of $SL(n)$ contravariance (or covariance) is weakened to $O(n)$ contravariance (or covariance, respectively), then there might appear more valuations. For example, the map $Z: P \mapsto B_2^n$ for all $P \in \mathcal {P}_o ^n$ is an $O(n)$ contravariant and covariant Orlicz valuation for any $\varphi \in \Phi$, where $B_2^n$ is the unit ball in $\mathbb{R}^n$.

\section{Preliminaries and Notations}
Let $\mathbb{R}^n$ be the $n$-dimensional Euclidean space and $\{e_i\}_{i=1}^n$ be the standard basis of $\mathbb{R}^n$. The usual scalar product of two vectors $x,y \in \mathbb{R}^n$ shall be denoted by $x \cdot y$. The convex hull of a set $A \subset \mathbb{R}^n$ will be denoted by $[A]$.

A hyperplane $H$ through the origin with a normal vector $u$ is defined by $\{ x \in\mathbb{R}^n : x \cdot u = 0 \}$. Furthermore define $H^- := \{ x \in\mathbb{R}^n : x \cdot u \leq 0 \}$ and $H^+ := \{ x \in\mathbb{R}^n : x \cdot u \geq 0 \}$.

Let $T^d = [o,e_1, \ldots,e_d]$ for $1 \leq d \leq n$. Denote by $\mathcal {T}_o^n$ the set of simplices containing the origin as one of their vertices. Define $\mathcal{P}_1 := \mathcal {T}_o^n$ and $\mathcal{P}_i := \mathcal{P}_{i-1} \cup \{ P_1 \cup P_2 \in \mathcal {P}_o^n: P_1,P_2 \in \mathcal{P}_{i-1} ~\text{with~disjoint~relative~interiors}\}$ recursively. Note that for any $P \in \mathcal {P}_o^n$, there exists an $i$ such that $P \in \mathcal{P}_{i}$.

Let $H \subset \mathbb{R}^n$ be a hyperplane through the origin. For any $P \in \mathcal{P}_{i}$, $i \geq 1$, we also have
\begin{align}\label{303}
P \cap H \in \mathcal{P}_{i}.
\end{align}
Indeed, for any $T \in \mathcal {T}_o^n$, we have $T \cap H \in \mathcal {T}_o^n$. Assume that for any $P \in \mathcal{P}_{i-1}$, $i \geq 2$, we have $P \cap H \in \mathcal{P}_{i-1}$. Then for any $P = P_1 \cup P_2$, where $P_1,P_2 \in \mathcal{P}_{i-1}$ have disjoint relative interiors, we have
$$P \cap H = (P_1 \cap H) \cup (P_2 \cap H).$$
If $P_1 \cap H$ and $P_2 \cap H$ have disjoint relative interiors, then $P \cap H \in \mathcal{P}_{i}$. If $P_1 \cap H$ and $P_2 \cap H$ have joint relative interiors, then only two possibilities could happen: $(P_1 \cap H) \subset (P_2 \cap H)$ and $(P_2 \cap H) \subset (P_1 \cap H)$. For both possibilities, we have $P \cap H \in \mathcal{P}_{i-1} \subset \mathcal{P}_i$.

The \emph{support function} of a convex body $K$ is defined by
\begin{align*}
h_K(x) = \max \{ x \cdot y : y \in K\}
\end{align*}
for any $x \in \mathbb{R}^n$. It is easy to see that
\begin{align}\label{210}
h_{\lambda K} = \lambda h_K
\end{align}
for any $\lambda \geq 0$ and any convex body $K$. The support function is sublinear, i.e., it is homogeneous,
\begin{align*}
h_K(\lambda x) = \lambda h_K(x)
\end{align*}
for any $x \in \mathbb{R}^n$, $\lambda \geq 0$ and subadditive,
\begin{align*}
h_K(x+y) \leq h_K(x) + h_K(y)
\end{align*}
for any $x,y \in \mathbb{R}^n$. The support function is also continuous on $\mathbb{R}^n$ by its convexity. A convex body is uniquely determined by its support function, and for any sublinear function $h$, there exists a convex body $K$ such that $h_K = h$.

Let $\varphi \in \Phi$. The \emph{Orlicz centroid body} of $K \in \mathcal {K}_o ^n$ (actually for any star body) introduced by Lutwak, Yang, and Zhang \cite{LYZ10b} is defined by
\begin{align*}
  h_{\Gamma_ \varphi K} (x) = \inf \{ \lambda >0 : \frac{1}{|K|} \int_K \varphi(\frac{|x \cdot y|}{\lambda}) dy \leq 1\}
\end{align*}
for any $x \in \mathbb{R}^n \setminus \{ o \}$, and $h_{\Gamma_ \varphi K} (o) = 0$. Lutwak, Yang, and Zhang \cite{LYZ10b} show that the Orlicz centroid body operator is $SL(n)$ covariant, i.e.,
\begin{align*}
  \Gamma_ \varphi \psi K = \psi \Gamma_ \varphi  K
\end{align*}
for any star body $K$ and $\psi \in SL(n)$.

We can also define the \emph{Orlicz moment body} of $K \in \mathcal {K}_o ^n$ by
\begin{align}\label{301}
  h_{M_ \varphi K} (x) = \inf \{ \lambda >0 : \int_K \varphi(\frac{|x \cdot y|}{\lambda}) dy \leq 1\}
\end{align}
for any $x \in \mathbb{R}^n \setminus \{ o \}$, and $h_{M_ \varphi K} (o) = 0$. It is easy to see that the Orlicz moment body operator is also $SL(n)$ covariant. When $\varphi(t) = t^p$, $t \geq 0$, for some $p \geq 1$, it is the $L_p$ moment body which was first characterized as an $SL(n)$ covariant and $\frac{n}{p}+1$ homogeneous $L_p$ Minkowski valuation by Ludwig \cite{Lud05}. Also see Theorem \ref{lpco}, where $M_p^+ K$ and $M_p^- K$ are the asymmetric $L_p$ moment bodies of $K$ (the absolute value of $x \cdot y$ in the definition (\ref{301}) is changed to the positive and negative part, respectively).

Let $\varphi \in \Phi$. Also introduced by Lutwak, Yang, and Zhang \cite{LYZ10a}, the \emph{Orlicz projection body} of a convex body $K$ containing the origin in its interior is defined by
\begin{align*}
  h_{\Pi_ \varphi K} (x) = \inf \{ \lambda >0 : \int_{S^{n-1}} \varphi(\frac{|x \cdot u|}{\lambda h_K(u)}) dV_K(u) \leq 1\}
\end{align*}
for any $x \in \mathbb{R}^n \setminus \{ o \}$, and $h_{\Pi_ \varphi K} (o) = 0$, where $dV_K(u) = h_K(u) dS_K(u)$ and $S_K(\cdot)$ is the surface area measure of $K$. When $\varphi(t) = t^p$, $t \geq 0$, for some $p \geq 1$, it is the $L_p$ projection body, denoted by $\Pi_p K$. When $p=1$, the ($L_1$) projection body operator $\Pi$ is defined on convex bodies (not necessarily containing the origin in their interior).
Lutwak, Yang, and Zhang \cite{LYZ10a} show that the Orlicz projection body operator is $SL(n)$ contravariant, i.e.,
\begin{align*}
  \Pi_ \varphi \psi K = \psi^{-t} \Pi_ \varphi  K
\end{align*}
for any convex body $K$ containing the origin in its interior and $\psi \in SL(n)$.

We can extend the Orlicz projection body operator to $\mathcal {P}_o ^n$ as Ludwig \cite{Lud05} did for $L_p$ cases. For $P \in \mathcal {P}_o ^n$,
\begin{align}\label{302}
h_{\hat{\Pi}_\varphi P} (x) = \inf \{ \lambda >0 : \int_{S^{n-1} \setminus \mathcal {N}_o (P)} \varphi(\frac{|x \cdot u|}{\lambda h_P(u)}) dV_P(u) \leq 1\}
\end{align}
for any $x \in \mathbb{R}^n$, where $\mathcal {N}_o (P)$ is the set of all outer unit normals of facets, which contain the origin, of $P$. Using exactly the same proof in Lutwak, Yang, and Zhang \cite{LYZ10a}, we can see that $\hat{\Pi}_\varphi$ is also $SL(n)$ contravariant. When $\varphi(t) = t^p$, $t \geq 0$, for some $p \geq 1$, this operator was first characterized as an $SL(n)$ contravariant and $\frac{n}{p}-1$ homogeneous  $L_p$ Minkowski valuation by Ludwig \cite{Lud05}. Also see Theorem \ref{Lpcontra}, where $\hat{\Pi}_p^+ P$ and $\hat{\Pi}_p^- P$ are the asymmetric $L_p$ projection bodies of $K$ (the absolute value of $x \cdot y$ in the definition (\ref{302}) is changed to the positive and negative part, respectively), and $h_{\Pi_o P} = \frac{1}{2}h_{\Pi P} - h_{\hat{\Pi}^+ P}$.

We will define Orlicz addition on $[0,\infty)$ and collect some properties of Orlicz addition.

Let $\varphi \in \Phi$. We define \emph{the Orlicz sum} $a+_\varphi b$ by
\begin{align}\label{11a}
\varphi \left(\frac{a}{a+_\varphi b}\right) + \varphi\left(\frac{b}{a+_\varphi b}\right) =1
\end{align}
for $a,b \geq 0$. If both $a,b=0$, then $a+_\varphi b$ should be interpreted as $0$. Let $a=h_K (x)$ and $b=h_L (x)$ for some convex bodies $K,L$ and $x \in \mathbb{R}^n$, we see that this definition is equal to the definition (\ref{11}). Hence we will not distinguish these two definitions. Also $h_{K +_\varphi L}(x) = h_K(x)+_\varphi h_L(x)$ for any $x \in \mathbb{R}^n$.

By (\ref{4}), we get that there exists a $2$-dimensional convex body $M$ such that
\begin{align}\label{4a}
a+_\varphi b = h_M (a,b)
\end{align}
for arbitrary $a,b \geq 0$.

Orlicz addition $+_\varphi$ is homogeneous, i.e.,
\begin{align}\label{9}
\alpha a +_\varphi \alpha b = \alpha (a +_\varphi b)
\end{align}
for arbitrary $a,b \geq 0$, $\alpha \geq 0$ and continuous, i.e.,
\begin{align}\label{10}
a_i +_\varphi b_i \to a +_\varphi b,
\end{align}
provided that $a_i \to a$, $b_i \to b$, $a_i,b_i,a,b \geq 0$. The homogeneity of Orlicz addition follows directly from the definition. The continuity of Orlicz addition is proved by Gardner, Hug, and Weil \cite[Theorem 5.2]{GHW14} for the definition (\ref{3}) and Xi, Jin and Leng \cite[Lemma 3.1, Lemma 3.2]{XJL14} for the definitions (\ref{11}) and (\ref{11a}). We give a short proof here:

\begin{proof}
Since $a_i \to a$, $b_i \to b$, there exists $N>0$ such that $a_i <a+1, b_i <b+1$ when $i >N$. Then it is easy to see that $a_i +_\varphi b_i <(a+1) +_\varphi (b+1)$ when $i >N$. Hence the sequence $\{a_i +_\varphi b_i\}$ is uniformly bounded. For any convergent subsequence $\{a_{i_j} +_\varphi b_{i_j}\}$, set $ c := \lim\limits_{j \to \infty} a_{i_j} +_\varphi b_{i_j}$. Since $\varphi$ is continuous on $[0,\infty)$, we have
\begin{align*}
\lim\limits_{j \to \infty} \varphi\left(\frac{a_{i_j}}{a_{i_j} +_\varphi b_{i_j}}\right)
+ \varphi\left(\frac{b_{i_j}}{a_{i_j} +_\varphi b_{i_j}}\right)
 =\varphi \left( \frac{a}{c} \right) + \varphi \left( \frac{b}{c} \right) =1.
\end{align*}
Hence $c=a +_\varphi b$. Since any convergent subsequence of the uniformly bounded sequence $\{a_i +_\varphi b_i\}$ converges to $a+_\varphi b$, we get that $a_i +_\varphi b_i \to a+_\varphi b$. The continuity is established.
\end{proof}

Note that there exists $0 \leq \eta <1$ such that $\varphi^{-1} (0) = [0,\eta]$. If $\eta \neq 0$, $+_\varphi$ loses some good properties such as: the equality $a +_\varphi b = a +_\varphi c$ for $a,b,c \geq 0$ does not imply $b=c$. But we still have some good properties which we list here and which are easy to check:

\begin{pro}\label{pro2.1}
Let $\varphi \in \Phi$ and let $\varphi^{-1} (0) = [0,\eta]$ where $0 \leq \eta <1$. The following propositions hold true: \\
(\rmnum{1}) If $a +_\varphi b = a +_\varphi a$ for $a,b \geq 0$, then $a=b$. \\
(\rmnum{2}) If $a +_\varphi b = a +_\varphi c$ for $a,b,c \geq 0$ satisfying $a \leq b$, then $b = c$. \\
(\rmnum{3}) If $a +_\varphi b = c +_\varphi d$ for $a,b,c,d \geq 0$ satisfying $a \leq \min \{c,d\}$, then $b \geq \max \{c,d \}$. \\
(\rmnum{4}) Let $a +_\varphi b = c +_\varphi c$ for $a,b,c \geq 0$. If $b<c$ or $a >b$, then $a>c>b$. If $b>c$ or $a < b$, then $a <c <b$.\\
(\rmnum{5}) If $a +_\varphi b \leq a +_\varphi c$ for $a,b,c \geq 0$ satisfying $\frac{b}{a} > \eta$, then $b \leq c$. \\
(\rmnum{6}) If $a +_\varphi b = c +_\varphi d$ for $a,b,c,d \geq 0$ satisfying $\max \{\frac{b}{a},\frac{c}{a} \} \leq \eta$, then $a=d$.
\end{pro}

We will use the following result proved by Pearson \cite{Pea66} in a paper on topological semirings on $\mathbb{R}$ which was also used by Gardner, Hug, and Weil \cite{GHW13,GHW14} to show that Orlicz addition with the associative property will be $L_p$ Minkowski addition.
\begin{thm}[Pearson \cite{Pea66}]\label{thm2.1}
Let $f:[0,\infty]^2 \to [0,\infty]$ be a continuous function satisfying the following conditions: \\
(i) $f(rs,rt)=rf(s,t)$ for any $r,s,t \geq 0$, \\
(ii) $f(f(r,s),t) = f(r,f(s,t))$ for any $r,s,t \geq 0$. \\
Then either $f(s,t)=0$, or $f(s,t)=s$, or $f(s,t)=t$, or there exists $p$, $0<p \leq \infty$, such that
\begin{align*}
f(s,t)=(s^p +t^p)^{1/p},
\end{align*}
or there exists $-\infty \leq p<0$, such that
\begin{align*}
f(s,t)=\begin{cases}
        (s^p +t^p)^{1/p},~ & \text{if}~s>0~\text{and}~t>0,\\
        0,~& \text{if}~s=0~\text{or}~t=0,
        \end{cases}
\end{align*}
where $s,t \geq 0$. When $p= \infty$, we mean $f(s,t) = \max \{ s,t \}$. When $p= -\infty$, we mean $f(s,t) = \min \{ s,t \}$.
\end{thm}

\section{The Cauchy functional equation}
If a function $f:(0,\infty) \to \mathbb{R}$ satisfies the ordinary Cauchy functional equation
\begin{align}\label{206}
  f(x+y) = f(x)+f(y)
\end{align}
for any $x,y > 0$, and $f$ is bounded from below on some non-empty open interval $I \subset \mathbb{R}$, then there exists a constant $c \in \mathbb{R}$ such that
$$f(x) = cx$$
for any $x>0$.

In this section, we will give the solution to the Cauchy type functional equation,
\begin{align*}
  f(x+y) +_\varphi a = f(x) +_\varphi f(y)
\end{align*}
for any $x,y > 0$, where $a \geq 0$ is a constant, $\varphi \in \Phi$ and $+_\varphi$ is defined by (\ref{11a}), with the additional condition $f \geq 0$.

If $+_\varphi = +_p$ for some $1 \leq p < \infty$, we set $g(x) = f(x)^p - a^p$. The function $g$ satisfies the ordinary Cauchy functional equation (\ref{206}). Hence $$f(x) = (cx + a^p)^{1/p}$$
for some constant $c \in \mathbb{R}$. Since $f \geq 0$, we have $c \geq 0$. Now we only need to show the case $+_\varphi \neq +_p$ for any $p \geq 1$.

\begin{lem}\label{Cauchy}
If a function $f:(0,\infty) \to [0,\infty)$ satisfies
\begin{align}\label{52a}
  f(x+y) +_\varphi a = f(x) +_\varphi f(y)
\end{align}
for any $x,y > 0$, where $a \geq 0$ is a constant, $\varphi \in \Phi$ and $+_\varphi \neq +_p$ for any $p \geq 1$, then
\begin{align*}
  f(z) = a
\end{align*}
for any $z>0$.
\end{lem}
\begin{proof}
We will first prove that $f(z) <a$ is impossible for any $z>0$.

For any fixed $z>0$, assume that $f(z) < a$ (if $a =0$, we don't need to consider this case since $f(z) \geq 0$). We will show that
\begin{align}\label{cc2}
f(2^k z) < a
\end{align}
for any integer $k$, and the function $k \mapsto f(2^k z)$ decreases. It is trivial that (\ref{cc2}) holds for $k=0$. For any integer $k$, taking $x = y =2^ {k-1} z$ in (\ref{52a}), we get
\begin{align}\label{52a-1}
  f(2^k z) +_\varphi a = f(2^{k-1} z) +_\varphi f(2^{k-1} z).
\end{align}
For $k \geq 1$, assume that (\ref{cc2}) holds for $k-1$. Taking this assumption into (\ref{52a-1}), by Proposition \ref{pro2.1} (\rmnum{4}), we have
$$f(2^k z) < f(2^{k-1} z) <a.$$
Similarly, for $k \leq -1$, assume that (\ref{cc2}) holds for $k+1$. Taking this assumption into (\ref{52a-1}), by Proposition \ref{pro2.1} (\rmnum{4}), we have
$$f(2^{k+1} z) < f(2^k z) <a.$$
Thus, the desired result has been shown.

Since the function $k \mapsto f(2^k z)$ is nonnegative and decreases, the limit exists when $k \to \infty$. Denote this limit by $b$. Then $0 \leq b <a$. Taking $k \to \infty$ in (\ref{52a-1}), we have $b +_\varphi a = b +_\varphi b$. By Proposition \ref{pro2.1} (\rmnum{1}), we have $b=a$. It is a contradiction to $b <a$. So
\begin{align}\label{209}
f(z) \geq a
\end{align}
for any $z > 0$.

Next, we will show that $f(z) > a$ is also impossible for any $z>0$.

For any fixed $z>0$, assume that $f(z) > a$. Using the similar methods in the case $f(z) < a$, we get
\begin{align*}
f(2^k z) > a
\end{align*}
for any integer $k$, and the function $k \mapsto f(2^k z)$ increases. Then we obtain that
\begin{align*}
\lim\limits_{k \to \infty} f(2^k z) = \infty.
\end{align*}
Indeed, if $\lim\limits_{k \to \infty} f(2^k z)$ is a finite number, denote by $b$. It is easy to see that $b >a$. Taking $k \to \infty$ in (\ref{52a-1}), we have $b +_\varphi a = b +_\varphi b$. By Proposition \ref{pro2.1} (\rmnum{1}), we have $b=a$. It is a contradiction to $b > a$.

For any $0 < x_1 < x_2$, taking $x+y=x_2, x=x_1$ in (\ref{52a}), combining with $a \leq \min \{f(x_1),f(x_2-x_1)\}$ (the inequality (\ref{209})) and Proposition \ref{pro2.1} (\rmnum{3}), we obtain that
\begin{align*}
f(x_2) \geq \max \{f(x_1),f(x_2-x_1)\} \geq f(x_1)
\end{align*}
for any $0 < x_1 < x_2$. Hence the function $f(x)$ increase. So the limit exists when $x \to 0^+$. Taking $x, y \to 0^+$ in (\ref{52a}), by the continuity of Orlicz addition (\ref{10}) and Proposition \ref{pro2.1} (\rmnum{1}), we have
\begin{align}\label{92}
  \lim\limits_{x \to 0^+}f(x) = a.
\end{align}
Hence, for arbitrary $x_0 \geq 0$, taking $x=x_0$, $y \to 0 ^+$ in (\ref{52a}), combining with (\ref{92}), the continuity of Orlicz addition (\ref{10}) and Proposition \ref{pro2.1} (\rmnum{2}), we get
\begin{align*}
\lim\limits_{x \to x_0 ^+}f(x) = f(x_0).
\end{align*}
Similarly, taking $x+y=x_0$, $x \to x_0 ^-$ in (\ref{52a}), we get
\begin{align*}
\lim\limits_{x \to x_0 ^-}f(x) = f(x_0).
\end{align*}
These show that the function $f(x)$ is continuous for any $x >0$. Combining with (\ref{92}), we have
$$f((0,\infty)) = (a,\infty).$$

Since (\ref{52a}) holds for any $x,y >0$, we get that
\begin{align*}
  f(\alpha) +_\varphi f(\beta) = f(\alpha+\beta) +_\varphi a= f(\gamma+\eta)+_\varphi a = f(\gamma) +_\varphi f(\eta)
\end{align*}
for any $\alpha,\beta,\gamma,\eta >0$ satisfying with $\alpha+\beta=\gamma+\eta $.
Combining with $f(\alpha) +_\varphi f(\alpha) = \frac{f(\alpha)}{\varphi^{-1} (\frac{1}{2})}$ for any $\alpha >0$ and the homogeneity of Orlicz addition (\ref{9}), for any $\alpha = \alpha_1 + \alpha_2$, $\beta = \beta_1 + \beta_2$, $\alpha_1, \alpha_2, \beta_1, \beta_2 >0$, we have
\begin{align}\label{53}
  \frac{f(\alpha)}{\varphi^{-1} (\frac{1}{2})} +_\varphi \frac{f(\beta)}{\varphi^{-1} (\frac{1}{2})}
  &= (f(\alpha) +_\varphi f(\alpha)) +_\varphi (f(\beta) +_\varphi f(\beta)) \nonumber \\
  &= (f(2 \alpha_1) +_\varphi f(2 \alpha_2)) +_\varphi (f(2\beta_1) +_\varphi f(2\beta_2)),
\end{align}
and
\begin{align}\label{54}
  \frac{1}{\varphi^{-1} (\frac{1}{2})} (f(\alpha)+_\varphi f(\beta))
    &= \frac{1}{\varphi^{-1} (\frac{1}{2})} (f(\alpha_1 + \beta_1) +_\varphi f(\alpha_2 + \beta_2)) \nonumber \\
    &= (f(2\alpha_1) +_\varphi f(2\beta_1)) +_\varphi (f(2\alpha_1) +_\varphi f(2\beta_1)).
\end{align}
Since $f((0,\infty)) = (a,\infty)$, we can choose $\alpha_1,\alpha_2,\beta_1,\beta_2$ such that $f(2\alpha_1)$, $f(2\beta_1)$, $f(2\alpha_1)$, $f(2\beta_1)$ are arbitrary real numbers larger than $a$. Hence, the relations (\ref{53}), (\ref{54}) and the homogeneity and of Orlicz addition (\ref{9}) imply that
\begin{align*}
  (r +_\varphi s) +_\varphi (w +_\varphi t) = (r +_\varphi w) +_\varphi ( s +_\varphi t)
\end{align*}
for any $r,s,w,t >0$. By the continuity of Orlicz addition, we have (letting $w \to 0^+$)
\begin{align*}
  (r +_\varphi s) +_\varphi t = r +_\varphi (s +_\varphi t)
\end{align*}
for any $r,s,t \geq 0$. By (\ref{4a}), we have
\begin{align*}
h_M(h_M(r,s),t) = h_M(r,h_M(s,t)),
\end{align*}
where $M$ is a $2$-dimensional convex body independent to $r,s,t$. Now combining with Theorem \ref{thm2.1} and the convexity of $h_M$, we obtain that there exists a real $p \geq 1$ such that
\begin{align*}
h_M (s,t) = (s^p+t^p)^{1/p}
\end{align*}
for any $s,t \geq 0.$
Thus, from (\ref{11}), (\ref{11a}), (\ref{4}), (\ref{4a}) and the definition of $L_p$ Minkowski addition (\ref{205}), we conclude that $+_\varphi= +_p$ which contradict to the condition of this theorem.

Hence $f(z) =c$ for any $z >0$.
\end{proof}

\section{$SL(n)$ contravariant valuations}
We call a valuation $Z$ simple if $Z$ vanishes on lower dimensional convex bodies.
In this section, we first show that any $SL(n)$ contravariant Orlicz valuation for $\varphi \in \Phi$ is simple on $\mathcal{T}_o ^n$ when $+_\varphi$ is not Minkowski addition. Here a valuation on $\mathcal{T}_o ^n$ means that the relation (\ref{16}) holds for $K,L, K \cup L, K \cap L \in \mathcal{T}_o ^n$.

\begin{lem}\label{lem3.1}
Let $n \geq 3$. If $Z : \mathcal{T}_o ^n \to \langle \mathcal {K}_o ^n, +_\varphi \rangle$ is an $SL(n)$ contravariant Orlicz valuation for $\varphi \in \Phi$, and $+_\varphi$ is not Minkowski addition, then $Z$ is simple.
\end{lem}
\begin{proof}
Let $T \in \mathcal{T}_o ^n$ and $\dim T =d <n$. By the $SL(n)$ contravariance of $Z$, we can assume (w.l.o.g.) that the linear space of $T$ is $ \text{span} \{ e_1, \ldots, e_d \}$, the linear space spanned by $ \{ e_1, \ldots, e_d \}$. Let $\psi := \left[ {\begin{array}{*{20}{c}}
I&A \\
0&B
\end{array}} \right] \in SL(n)$, where $I \in \mathbb{R}^{d \times d}$ is the identity matrix, $A \in \mathbb{R}^{d \times (n-d)}$ is an arbitrary matrix, $B \in \mathbb{R}^{(n-d) \times (n-d)}$ is a matrix with $\text{det} B = 1$, $0 \in \mathbb{R}^{(n-d)\times d}$ is the zero matrix. Also, let $x = \left( {\begin{array}{*{20}{c}}
{x'}\\
{x''}
\end{array}} \right) \in \mathbb{R}^{d \times (n-d)}$ and $x'' \neq 0$. Then $\psi T = T$. Combining with the $SL(n)$ contravariance of $Z$, we have
\begin{align*}
h_{ZT} (x) = h_{Z \psi T} (x) = h_{ZT}(\psi ^{-1} x) = h_{ZT} \left( {\begin{array}{*{20}{c}}
{x' - AB^{-1}x''}\\
{B^{-1}x''}
\end{array}} \right).
\end{align*}

For $d \leq n-2$, we can choose an appropriate matrix $B$ such that $B^{-1}x''$ is any nonzero vector in $\text{Span} \{ e_{d+1},\ldots,e_n \}$. After fixing $B$ we can also choose an appropriate matrix $A$ such that $x' - AB^{-1}x''$ is any vector in $\text{Span} \{ e_1, \ldots, e_d \}$. So $h_{ZT} (\cdot)$ is constant on a dense set of $\mathbb{R}^n$. By the continuity of the support function, we get $h_{ZT} = 0$.

In the case $d = n-1$ we have $B=1$. Then we can choose $A$ such that $x' - AB^{-1}x''=0$ and $h_{ZT} (x) = h_{ZT} (x_n e_n)$, where $x_n$ is the $n$-th coordinate of $x$. Next we want to show that $h_{Z(sT^{n-1})} (e_n)= 0$ for any $s > 0$.

For $0 < \lambda < 1$, we denote by $H_\lambda$ the hyperplane through the origin with a normal vector $(1-\lambda) e_1- \lambda e_2$. Since $Z$ is an Orlicz valuation,
\begin{align*}
h_{Z(sT^{n-1})} (e_n) +_\varphi h_{Z(sT^{n-1} \cap H_\lambda)} (e_n) = h_{Z(sT^{n-1}\cap H_\lambda ^-)} (e_n) +_\varphi h_{Z(sT^{n-1} \cap H_\lambda ^+)} (e_n).
\end{align*}
From the conclusion above for $d=n-2$, we get
\begin{align*}
h_{Z(sT^{n-1})} (e_n) = h_{Z(sT^{n-1}\cap H_\lambda ^-)} (e_n) +_\varphi h_{Z(sT^{n-1} \cap H_\lambda ^+)} (e_n).
\end{align*}
Define $\psi _1 \in SL(n)$ by
$$\psi _1 e_1 = \lambda e_1 + (1-\lambda) e_2,~\psi _1 e_2 = e_2,~\psi _1 e_n = \frac{1}{\lambda} e_n,~\psi _1 e_i = e_i,~\text{for}~3 \leq i \leq n-1.$$
Also define $\psi _2 \in SL(n)$ by
$$\psi _2 e_1 = e_1,~\psi _2 e_2 = \lambda e_1 + (1-\lambda) e_2,~\psi _2 e_n = \frac{1}{1-\lambda} e_n,~\psi _2 e_i = e_i,~\text{for}~3 \leq i \leq n-1.$$
So $sT^{n-1}\cap H_\lambda ^- = \psi _1 sT^{n-1}$, $sT^{n-1}\cap H_\lambda ^+ = \psi _2 sT^{n-1}$. By the $SL(n)$ contravariance of $Z$, we obtain
\begin{align*}
h_{Z(sT^{n-1})} (e_n) &= h_{Z(\psi _1 sT^{n-1})} (e_n) +_\varphi h_{Z(\psi _2 sT^{n-1})} (e_n), \\
&= h_{Z( sT^{n-1})} (\psi _1 ^{-1} e_n) +_\varphi h_{Z (sT^{n-1})} (\psi _2 ^{-1} e_n), \\
&= h_{Z( sT^{n-1})} (\lambda e_n) +_\varphi h_{Z (sT^{n-1})} ((1-\lambda) e_n).
\end{align*}

If $h_{Z(sT^{n-1})} (e_n) \neq 0$, by the homogeneity of the support function, the definition of Orlicz addition (\ref{11a}) and the continuity of $\varphi$ on $[0,\infty]$, we have
\begin{align}\label{201}
\varphi (\lambda) + \varphi (1-\lambda) = 1
\end{align}
for arbitrary $0 \leq \lambda \leq 1$.
Since $\varphi \in \Phi$,
\begin{align}\label{202}
\varphi (\lambda) \leq (1-\lambda) \varphi(0) + \lambda \varphi(1) = \lambda
\end{align}
for any $0 \leq \lambda \leq 1$. Combining (\ref{201}) with (\ref{202}), we get that $\varphi$ is linear on $[0,1]$. By (\ref{205}), (\ref{11}) and (\ref{11a}), we get that $+_\varphi$ is Minkowski addition, a contradiction.
Hence, $h_{Z(sT^{n-1})} (e_n)= 0$ for any $s > 0$.

Combining with the homogeneity of the support function, we get that $h_{Z(sT^{n-1})} (x) = h_{Z(sT^{n-1})} (x_ne_n)=0$. Since $Z$ is $SL(n)$ contravariant, we get that $h_{ZT} =0$ for $\dim T \leq n-1$.
\end{proof}

Now we use the Cauchy functional equation (\ref{52a}) to give the main results in the contravariant case. Since $ZP = \{ o \}$ for all $P \in \mathcal{P}_o ^n$ is an $SL(n)$ contravariant Orlicz valuation for any $\varphi \in \Phi$, we only need to prove the necessary condition of Theorem \ref{thm1.3}.
\begin{thm}\label{thm3.2}
Let $n \geq 3$ , $\varphi \in \Phi$ and $+_\varphi \neq +_p$ for any $p \geq 1$. If $Z : \mathcal{P}_o ^n \to \langle \mathcal {K}_o ^n, +_\varphi \rangle$ is an $SL(n)$ contravariant Orlicz valuation for $\varphi$, then
\begin{align}\label{h1}
ZP = \{ o \}
\end{align}
for all $P \in \mathcal{P}_o ^n$.
\end{thm}
\begin{proof}
For $0<\lambda <1$, let $H_\lambda$ denote the hyperplane through the origin with a normal vector $(1-\lambda) e_1- \lambda e_2$. Since $Z$ is a valuation, for any $x \in \mathbb{R}^n$, $s>0$, we have
\begin{align*}
h_{Z(sT^{n})} (x) +_\varphi h_{Z(sT^{n} \cap H_\lambda)} (x) = h_{Z(sT^{n}\cap H_\lambda ^-)} (x) +_\varphi h_{Z(sT^{n} \cap H_\lambda ^+)} (x).
\end{align*}
By Lemma \ref{lem3.1}, $h_{Z(sT^{n} \cap H_\lambda)} (x) = 0$. Thus,
\begin{align*}
h_{Z(sT^{n})} (x)= h_{Z(sT^{n}\cap H_\lambda ^-)} (x) +_\varphi h_{Z(sT^{n} \cap H_\lambda ^+)} (x).
\end{align*}
Define $\psi_1 \in SL(n)$ by
$$\psi_1 e_1 = (\frac{1}{\lambda})^{1/n} (\lambda e_1 + (1-\lambda) e_2),~\psi_1 e_2 = (\frac{1}{\lambda})^{1/n} e_2,~\psi_1 e_i = (\frac{1}{\lambda})^{1/n} e_i,~\text{for}~3 \leq i \leq n.$$
Also define $\psi_2 \in SL(n)$ by
$$\psi_2 e_1 = (\frac{1}{1-\lambda})^{1/n} e_1,~\psi_2 e_2 = (\frac{1}{1-\lambda})^{1/n} (\lambda e_1 + (1-\lambda) e_2),$$
$$\psi_2 e_i = (\frac{1}{1-\lambda})^{1/n} e_i,~\text{for}~3 \leq i \leq n.$$
So $sT^{n}\cap H_\lambda ^- = \psi_1 \lambda^{1/n} sT^{n}$, $sT^{n}\cap H_\lambda ^+ = \psi_2 (1-\lambda)^{1/n}sT^{n}$. By the $SL(n)$ contravariance of $Z$, we get
\begin{align}\label{7}
h_{Z(sT^{n})} (x)
=h_{Z(\lambda^{1/n} sT^{n})} (\psi_1  ^{-1} x) +_\varphi h_{Z((1-\lambda)^{1/n}sT^{n})} (\psi_2^{-1}x),
\end{align}
where $x=(x_1, \cdots, x_n)^t$, $\psi_1 ^{-1} x = \lambda^{1/n} (\frac{1}{\lambda} x_1, \frac{\lambda -1 }{\lambda}x_1 +x_2, x_3,\cdots,x_n)^t$, $\psi_2 ^{-1} x = (1-\lambda)^{1/n}(x_1 - \frac{\lambda}{1-\lambda} x_2, \frac{1}{1-\lambda}x_2,x_3,\cdots,x_n)^t$.
If we choose $x =\pm e_n$ in (\ref{7}), then
\begin{align}\label{8}
h_{Z(sT^{n})} (\pm e_n)
=h_{\lambda ^{1/n} Z(\lambda ^{1/n} sT^{n})} (\pm e_n) +_\varphi h_{(1- \lambda)^{1/n}Z((1-\lambda)^{1/n}sT^{n})} (\pm e_n)
\end{align}
for $0 < \lambda <1$, and $s>0$.
Taking $\lambda = \frac{\lambda_1}{\lambda_2}$, $0 < \lambda_1 < \lambda_2$ and $s = \lambda_2 ^{1/n}$ in (\ref{8}), with (\ref{210}) and the homogeneity of Orlicz addition (\ref{9}), we get
\begin{align}\label{1}
h_{\lambda_2^{1/n}Z(\lambda_2^{1/n}T^{n})} (\pm e_n)
= h_{\lambda_1^{1/n}Z(\lambda_1^{1/n}T^{n})} (\pm e_n) +_\varphi h_{(\lambda_2-\lambda_1)^{1/n}Z((\lambda_2-\lambda_1)^{1/n}T^{n})} (\pm e_n)
\end{align}
for arbitrary $0 < \lambda_1 < \lambda_2$, $s>0$.

Define $f(\lambda) := h_{\lambda^{1/n}Z(\lambda^{1/n}T^{n})} (\pm e_n)$ for $\lambda >0$. (\ref{1}) shows that $f$ satisfies the Cauchy functional equation (\ref{52a}) with $a=0$. Hence, Lemma \ref{Cauchy} shows that $h_{\lambda^{1/n}Z(\lambda^{1/n}T^{n})} (\pm e_n) = 0$ for any $\lambda >0$. That means $h_{Z(sT^{n})} (\pm e_n) = 0$ for any $s>0$. By the $SL(n)$ contravariance of $Z$, we get that $h_{Z(sT^{n})} (\pm e_i) = 0$ for $1\leq i \leq n$. Since the support function is sublinear, we get that
$$h_{Z(sT^{n})} (x) = 0$$
for any $x \in \mathbb{R}^n$, $\lambda >0$. Hence $Z(sT^n)= \{ o \}$ for any $s>0$.

By the $SL(n)$ contravariance of $Z$ and Lemma \ref{lem3.1}, (\ref{h1}) holds true for any simplex in $\mathcal {T}_o^n = \mathcal{P}_{1}$.
Assume that (\ref{h1}) holds on $\mathcal{P}_{i-1}$, $i \geq 2$. For $P=P_1 \cup P_2 \in \mathcal{P}_{i}$, where $P_1,P_2 \in \mathcal{P}_{i-1}$ have disjoint relative interiors, by (\ref{303}), we have $P_1 \cap P_2 \in \mathcal{P}_{i-1}$. Hence we have
$$h_{Z(P_1 \cap P_2)} =0.$$
Therefore $h_{Z(P_1 \cup P_2)}$ is uniquely determined by (\ref{16}) and (\ref{11a}), namely,
\begin{align*}
h_{Z(P_1 \cup P_2)}= h_{ZP_1} +_\varphi h_{ZP_2} =0.
\end{align*}
Hence, we conclude that (\ref{h1}) holds on $\mathcal{P}_i$ inductively for any $i$. For any $P \in \mathcal {P}_o^n$, there exists an $i$ such that $P \in \mathcal{P}_{i}$. Thus (\ref{h1}) holds for all $P \in \mathcal {P}_o^n$.
\end{proof}

\section{$SL(n)$ covariant valuations}
If $K \cup L$ is convex, then
$$h_{K \cup L} = \max \{ h_K,h_L \}~ \text{and} ~h_{K \cap L} = \min \{ h_K,h_L \}.$$
Hence, it is easy to see that the identity operator and the reflection operator are $SL(n)$ covariant Orlicz valuations for any $\varphi \in \Phi$.
(For general $M$-addition, Mesikepp \cite{Mesikepp2013} showed that if $M \subset (-\infty,0]^2 \cup [0,\infty)^2$ is symmetric in the line $y=x$, then the identity operator is an valuation with respect to $M$-addition.)
As in the contravariant case, we only need to prove the necessary condition of Theorem \ref{thm1.4}.

The following Lemma can be found in Ludwig \cite{Lud05}, Haberl \cite{Hab12b} and Parapatits \cite{Par14b}. For completeness, we give a proof here.

\begin{lem}\label{lem4.1}
Let $n \geq 2$. If a map $Z : \mathcal{P}_o ^n \to \mathcal{K}_o ^n$ is $SL(n)$ covariant, then $ZP \subset \text{lin}~ P$, and
\begin{align*}
  h_{ZP} (x) = h_{ZP} (\pi_P x), ~x \in \mathbb{R}^n
\end{align*}
for any $P \in \mathcal{P}_o ^n$, where $\pi_P x$ is the orthogonal projection of $x$ onto linear hull of $P$.
\end{lem}
\begin{proof}
Let $P \in \mathcal{P}_o ^n$. Since $Z$ is $SL(n)$ covariant, we can assume (w.l.o.g.) that the linear space of $P$ is $ \text{span} \{ e_1, \cdots, e_d \}$, the linear space spanned by $ \{ e_1, \cdots, e_d \}$. If $d=n$, the statement is trivial. Now let $d <n$. Denote $\psi := \left[ {\begin{array}{*{20}{c}}
I_d &A \\
0& I_{n-d}
\end{array}} \right] \in SL(n)$, where $I_d \in \mathbb{R}^{d \times d}, I_{n-d} \in \mathbb{R}^{(n-d) \times (n-d)}$ are the identity matrixes, $A \in \mathbb{R}^{d \times (n-d)}$ is an arbitrary \text{matrix}, $0 \in \mathbb{R}^{(n-d)\times d}$ is the zero matrix. Also, let $x = \left( {\begin{array}{*{20}{c}}
{x'}\\
{x''}
\end{array}} \right) \in \mathbb{R}^{d \times (n-d)}$, and $x' \neq 0$. Then $\psi P = P$. Combining with the $SL(n)$ covariance of $Z$, we have
\begin{align*}
h_{ZP} (x) = h_{Z \psi P} (x) = h_{ZP}(\psi ^t x) = h_{ZP} \left( {\begin{array}{*{20}{c}}
{x'}\\
{A^t x' +  x''}
\end{array}} \right).
\end{align*}
We can choose an appropriate matrix $A$ such that $A^t x' + x''=0$. Hence $h_{ZP} (x) = h_{ZP} (x')$ when $x' \neq 0$. With the continuity of the support function, we obtain the desired result.
\end{proof}

Although the identity operator and the reflection operator are \text{Orlicz} valuations, unlike in the $L_p$ cases, we will show that the Orlicz difference body operator is not an Orlicz valuation for $\varphi \in \Phi$ when $+_\varphi \neq +_p$.

\begin{lem}\label{lem4.3}
For any $a,b >0$, $\varphi \in \Phi$, if $+_\varphi \neq +_p$ for any $p \geq 1$, then the map $Z : \mathcal{P}_o ^n \to \langle \mathcal {K}_o ^n, +_\varphi \rangle$ defined by $ZP = aP +_\varphi (-bP)$ is not an Orlicz valuation for $\varphi$.
\end{lem}
\begin{proof}
We prove the assertion by contradiction. For any $s,t \geq 0$, we will briefly write $[-s,t] := [-se_1,te_1]$. By the definition of the Orlicz addition (\ref{11}) and (\ref{11a}),
\begin{align}\label{45}
Z[-s,t] = a [-s,t] +_\varphi b[-t,s] = [-(as +_\varphi bt), (at +_\varphi bs)].
\end{align}
Assume that $Z$ is a valuation, we have
\begin{align}\label{46}
  Z[-s,t_2] +_\varphi Z[0,t_1] = Z[-s,t_1] +_\varphi Z[0,t_2]
\end{align}
for any $s,t_1,t_2 \geq 0$. Combining (\ref{45}) with (\ref{46}), we get that
\begin{align*}
  &[-(as +_\varphi bt_2), (at_2 +_\varphi bs)] +_\varphi [-bt_1,at_1] \\
  & \qquad = [-(as +_\varphi bt_1), (at_1 +_\varphi bs)] +_\varphi [-bt_2,at_2].
\end{align*}
Using the definition of Orlicz addition again, we obtain that
\begin{align*}
  (as +_\varphi bt_2) +_\varphi (bt_1) = (as +_\varphi bt_1) +_\varphi (bt_2).
\end{align*}
Since Orlicz addition is commutative, combining with (\ref{4a}), we have
\begin{align*}
h_M(h_M(bt_2,as),bt_1) = h_M(bt_2,h_M(as,bt_1)),
\end{align*}
where $M$ is a $2$-dimensional convex body independent of the numbers $bt_2,as$ and $bt_1$. Now combining with Theorem \ref{thm2.1} and the convexity of $h_M$, we obtain that there exists a real $p \geq 1$ such that
\begin{align*}
h_M (s,t) = (s^p+t^p)^{1/p}
\end{align*}
for any $s,t \geq 0.$
By (\ref{11}), (\ref{11a}), (\ref{4}), (\ref{4a}) and the definition of $L_p$ Minkowski addition (\ref{205}), we get that $+_\varphi= +_p$, a contradiction.
\end{proof}

Let $n \geq 3$. For $1 \leq d \leq n$, we will show some properties of the function $h_{ZsT^d}$ on the first coordinate axis in $\mathbb{R}^n$ where $Z$ is an $SL(n)$ covariant Orlicz valuation.

\begin{lem}\label{lem4.5}
Let $n \geq 3$. If $Z : \mathcal{P}_o ^n \to \langle \mathcal {K}_o ^n, +_\varphi \rangle$ is an $SL(n)$ covariant Orlicz valuation for $\varphi \in \Phi$, and $+_\varphi \neq +_p$ for any $p \geq 1$, then
\begin{align}\label{e1}
h_{ZsT^d} ( \pm e_1) = sh_{ZT^d} ( \pm e_1)
\end{align}
for any $1 \leq d \leq n$, $s>0$, and
\begin{align}\label{e2}
h_{ZT^1} ( \pm e_1) = \cdots = h_{ZT^n} ( \pm e_1).
\end{align}
Furthermore, either $h_{ZT^1}(e_1) = 0$ or $h_{ZT^1}(-e_1) = 0$.
\end{lem}
\begin{proof}
Since $Z$ is $SL(n)$ covariant, $h_{ZsT^d} (\pm e_i) = h_{ZsT^d} (\pm e_1)$ for any $1 \leq i \leq d$. Let $a_d : = h_{ZT^d} (e_1)$, $b_d : = h_{ZT^d} (-e_1)$.

\vskip 5pt
We want first to show that (\ref{e1}) holds for $d=n$ and that $h_{ZT^{n-1}} ( \pm e_1) = h_{ZT^n} ( \pm e_1)$.

By the $SL(n)$ covariance of $Z$, we see that
\begin{align}\label{70}
h_{Z(s\hat{T}^{n-1})} (e_n)= s a_{n-1}
\end{align}
for any $s >0$, where $\hat{T}^{n-1}=[o,e_1,e_3,\cdots,e_n]$.

For $0 < \lambda <1$, define $H_\lambda$, $\psi _1,\psi _2$ as in Theorem \ref{thm3.2}. Since $Z$ is an Orlicz valuation,
\begin{align*}
h_{Z(sT^{n})} +_\varphi h_{Z(sT^n \cap H_\lambda)}
=h_{Z(sT^n \cap H_\lambda ^-)} +_\varphi h_{Z(sT^n \cap H_\lambda ^+)}.
\end{align*}
Then, by the $SL(n)$ covariance of $Z$, we have
\begin{align}\label{60}
h_{Z(sT^{n})} (x) +_\varphi h_{Z(\lambda^{1/n}s\hat{T}^{n-1})} (\psi_1 ^t x)
=h_{Z(\lambda^{1/n} sT^{n})} (\psi_1 ^t x) +_\varphi h_{Z((1-\lambda)^{1/n}sT^{n})} (\psi_2 ^t x),
\end{align}
where $x=(x_1, \cdots, x_n)^t$, $\psi_1 ^t x = \lambda^{-1/n} (\lambda x_1 + (1-\lambda)x_2, x_2, x_3,\cdots,x_n)^t$ and $\psi_2 ^t x = (1-\lambda)^{-1/n}(x_1, \lambda x_1+(1-\lambda)x_2,x_3,\cdots,x_n)^t$. Taking $x =e_n$ in (\ref{60}), we have
\begin{align}\label{51}
&h_{Z(sT^{n})} (e_n) +_\varphi h_{\lambda^{-1/n}Z(\lambda^{1/n}s\hat{T}^{n-1})} (e_n) \nonumber \\
&=h_{\lambda^{-1/n}Z(\lambda^{1/n} sT^{n})} (e_n) +_\varphi h_{(1-\lambda)^{-1/n}Z((1-\lambda)^{1/n}sT^{n})} (e_n)
\end{align}
for any $0 < \lambda <1$, $s>0$. Also taking $\lambda = \frac{\lambda_1}{\lambda_2}$, $0 < \lambda_1 < \lambda_2$ and $s = \lambda_2 ^{1/n}$ in (\ref{51}), with (\ref{70}), (\ref{210}) and the homogeneity of Orlicz addition (\ref{9}), we get
\begin{align}\label{52}
  &h_{\lambda_2^{-1/n}Z(\lambda_2^{1/n}T^{n})} (e_n) +_\varphi a_{n-1} \nonumber \\
&= h_{\lambda_1^{-1/n}Z(\lambda_1^{1/n}T^{n})} (e_n) +_\varphi h_{(\lambda_2-\lambda_1)^{-1/n}Z((\lambda_2-\lambda_1)^{1/n}T^{n})} (e_n)
\end{align}
for any $0 < \lambda_1 < \lambda_2$.

Define $f(\lambda) := h_{\lambda^{-1/n}Z(\lambda^{1/n}T^{n})} (e_n)$ for $\lambda >0$. Hence (\ref{52}) implies that $f$ satisfies the Cauchy functional equation (\ref{52a}) with $a=a_{n-1}$. By Lemma \ref{Cauchy}, we get
\begin{align}\label{56}
  h_{\lambda^{-1/n}Z(\lambda^{1/n}T^{n})} (e_n) = a_{n-1}
\end{align}
for any $\lambda >0$.
Similarly
\begin{align}\label{63}
h_{\lambda^{-1/n}Z(\lambda^{1/n}T^{n})} (-e_n) = b_{n-1}.
\end{align}
Hence (\ref{e1}) holds true for $d=n$.

\vskip 5pt
Now we consider the case $d \leq n-1$.

It is easy to see that $Z(sT^{d}) = s ZT^d$ by the $SL(n)$ covariance of $Z$. Then (\ref{e1}) holds for $d \leq n-1$.

If $d=1$, we write $[-s,t] := [-se_1,te_1]$ for any $s,t \geq 0$. By Lemma \ref{lem4.1}, we get that $Z[0, 1] = [-b_1,a_1]$. Since $Z$ is $SL(n)$ covariant, we have $Z[0, s] = -Z[-s,0] = sZ[0, 1]$ for any $s \geq 0$. Thus, $$Z[0, t] = t[-b_1,a_1] = [-b_1t,a_1t],~~Z[-s,0] = -sZ[0, 1] = [-a_1s,b_1s]$$
for any $s,t \geq 0$. Since $Z$ is a valuation, and $Z \{ o \} = \{ o \}$, we have
$$Z[-s,t] = Z[0, t] +_\varphi Z[-s,0] = [-(a_1s +_\varphi b_1t), (b_1s +_\varphi a_1t)].$$
It is similar to the relation (\ref{45}).
By the proof of Lemma \ref{lem4.3}, we have $a_1=0$ or $b_1=0$.

Hence, we will further assume that $b_1=h_{Z[0, e_1]}(-e_1)=0$. The case $a_1 = h_{Z[0, e_1]}(e_1)=0$ is similar.

If $d \leq n-1$, define $\psi _1 \in SL(n)$ by
$$\psi _1 e_1 = \lambda e_1 + (1-\lambda) e_2,~\psi _1 e_2 = e_2,~\psi _1 e_n = \frac{1}{\lambda} e_n,~\psi _1 e_i = e_i,~\text{for}~3 \leq i \leq n-1.$$
Also define $\psi _2 \in SL(n)$ by
$$\psi _2 e_1 = e_1,~\psi _2 e_2 = \lambda e_1 + (1-\lambda) e_2,~\psi _2 e_n = \frac{1}{1-\lambda} e_n,~\psi _2 e_i = e_i,~\text{for}~3 \leq i \leq n-1.$$
So $sT^{d}\cap H_\lambda ^- = \psi _1 sT^{d}$, $sT^{d}\cap H_\lambda ^+ = \psi _2 sT^{d}$. Denote $\hat{T}^{d-1}=[o,e_1,e_3,\cdots,e_d]$, then $sT^d \cap H_\lambda =\psi _1 s\hat{T}^{d-1}$.
Since $Z$ is an $SL(n)$ covariant Orlicz \text{valuation}, we obtain that
\begin{align}\label{55}
h_{ZT^d} (x) +_\varphi h_{Z\hat{T}^{d-1}} (\psi _1 ^t x)
= h_{ZT^d} (\psi _1 ^t x) +_\varphi h_{Z T^d} (\psi _2 ^t x) ,
\end{align}
where $x=(x_1, \cdots, x_n)^t$, $\psi_1 ^t x = (\lambda x_1 + (1-\lambda)x_2, x_2, x_3,\cdots,x_{n-1}, \frac{1}{\lambda} x_n)^t$, $\psi_2 ^t x = (x_1, \lambda x_1 + (1-\lambda)x_2,x_3,\cdots,x_{n-1},\frac{1}{\lambda}x_n)^t$.

Taking $x=e_1 + \cdots + e_d$ in (\ref{55}), combining with Proposition \ref{pro2.1} (\rmnum{1}), Lemma \ref{lem4.1} and the $SL(n)$ covariance of $Z$, we obtain that
$$h_{Z T^d} (e_1 + \cdots + e_d) = h_{Z \hat{T}^{d-1}} (e_1 + \cdots + e_d) = h_{Z T^{d-1}} (e_1 + \cdots + e_{d-1}).$$
Thus
\begin{align}\label{47}
  h_{Z T^d} (e_1 + \cdots + e_d)  = h_{Z T^{d-1}} (e_1 + \cdots + e_{d-1}) = \cdots = h_{ZT^1}(e_1) = a_1.
\end{align}
Similarly, taking $x=-(e_1 + \cdots + e_d)$ in (\ref{55}), we get that
\begin{align}\label{48}
  h_{Z T^d} (-(e_1 + \cdots + e_d))  = h_{Z T^{d-1}} (-(e_1 + \cdots + e_{d-1}))  \nonumber \\
   = \cdots = h_{ZT^1}(-e_1) = b_1 = 0.
\end{align}
Also, for $3 \leq d \leq n-1$, taking $x=e_d$ in (\ref{55}), we obtain that $a_{d} = a_{d-1}$ by Proposition \ref{pro2.1} (\rmnum{1}).
Thus, combining with (\ref{56}), we have
\begin{align}\label{57}
a_{n} = \cdots = a_{2}.
\end{align}
Similarly, taking $x= -e_d$ in (\ref{55}), combining with (\ref{63}), we get
\begin{align}\label{64}
b_{n} = \cdots = b_{2}.
\end{align}
Hence, we only need to prove that $a_1 = a_2$ and $b_1=b_2$ in the following part.

We first want to show that $b_2=0$ when $b_1 = 0$. Define $Z' : \mathcal {P}_o ^1 \to \mathcal {K}_o^1$ by $Z'I = [-h_{Z[I, e_2]}(-e_1) e_1, h_{Z[I,e_2]}(e_1) e_1]$ for $I \in \mathcal {P}_0 ^1$. Then $Z'$ is a valuation satisfying $Z'[0, se_1] = -Z'[-se_1,o] = sZ'[0, e_1]$ for $s \geq 0$. By the discussion of the case $d=1$, we have $h_{Z'[o, e_1]}(e_1) = 0$ or $h_{Z'[o, e_1]}(-e_1) =0$. Hence, we have
\begin{align}\label{58}
  a_2=0 ~\text{or}~b_2=0.
\end{align}

Since we have assumed that $b_1=0$, if also $a_1=0$, then by (\ref{47}) and (\ref{48}), $h_{ZT^2} (e_1+e_2) = h_{ZT^2} (-(e_1 + e_2)) = 0$.
Then by Lemma \ref{lem4.1}, we get $ZT^2 = [b_2(e_2-e_1),a_2(e_1-e_2)]$ since $a_2 = h_{ZT^2}(e_1), b_2 = h_{ZT^2}(-e_1)$. Also since $Z$ is $SL(n)$ covariant, $a_2 = b_2$. Combining with (\ref{57}), (\ref{64}) and (\ref{58}) we get
\begin{align*}
a_{n} = \cdots = a_{1} =0,~~
b_{n} = \cdots = b_{1} =0.
\end{align*}

Now we assume that $a_1 >0$.

Since the support function is subadditive, by (\ref{47}), we have
$$0 < a_1 = h_{ZT^2} (e_1 +e_2) \leq h_{ZT^2} (e_1) + h_{ZT^2} (e_2) = 2a_2.$$
Then by (\ref{58}), we have
\begin{align}\label{61}
  b_2=0.
\end{align}
Then
\begin{align}\label{59}
  a_2 = h_{ZT^2}(e_1) \leq h_{ZT^2}(e_1 + e_2) + h_{ZT^2} (-e_2) = a_1.
\end{align}

\vskip 5pt
Finally, we will use the sublinearity of $h_{ZT^3}$ to show that $a_2 \geq a_1$. Then combining with (\ref{57}), (\ref{64}), (\ref{61}), (\ref{59}) and the assumption $b_1=0$, we will get the equality (\ref{e2}), and the proof will be completed.

We will only prove for the case $n=3$ (the cases $n>3$ is similar and easier, by using (\ref{55}) instead of (\ref{60})).
For any $\alpha >0$, taking $n=3$, $s= (\frac{\alpha}{\lambda}) ^{1/n}$, $x=e_2$ in (\ref{60}), combining with (\ref{210}) and the homogeneity of Orlicz addition (\ref{9}), and $Z(s\hat{T}^2) = sZ\hat{T}^2$, we get
\begin{align}\label{91}
  &h_{(\frac{\alpha}{\lambda}) ^{-1/n}Z((\frac{\alpha}{\lambda}) ^{1/n}T^{3})} (e_2) +_\varphi h_{Z\hat{T}^2} ((1-\lambda)  e_1) \nonumber\\
&=h_{\alpha^{-1/n}Z(\alpha^{1/n}T^{3})} ((1-\lambda) e_1 + e_2) +_\varphi h_{(\frac{\alpha(1-\lambda)}{\lambda}) ^{-1/n} Z((\frac{\alpha(1-\lambda)}{\lambda}) ^{1/n}T^{3})} ((1-\lambda) e_2)
\end{align}
for any $0 < \lambda <1$ and $\alpha >0$. Since the function $\lambda \mapsto h_{\lambda^{-1/n}Z(\lambda^{1/n}T^3)} (e_2)$ is $0$-homogeneous for $\lambda >0$ (by (\ref{56}) and the $SL(n)$ covariance of $Z$), combining with $h_{ZT^{3}} (e_2)=a_3=a_2$, $h_{Z\hat{T}^2} (e_1) = h_{ZT^2} (e_1)= a_2$, and that support functions are homogeneous and continuous, by Proposition \ref{pro2.1} (\rmnum{2}), we get
\begin{align}\label{62}
  h_{\alpha^{-1/n}Z(\alpha^{1/n}T^{3})} (\lambda e_1 + e_2) = a_2
\end{align}
for any $\alpha >0$, $0 \leq \lambda \leq 1$.
Also taking $n=3$, $s= (\frac{\alpha}{1-\lambda}) ^{1/n}$, $x=e_1 + \mu e_3$, $0 < \mu < \lambda$ in (\ref{60}), we get
\begin{align*}
  &h_{(\frac{\alpha}{1-\lambda}) ^{-1/n}((\frac{\alpha}{1-\lambda}) ^{1/n}T^{3})} (e_1 + \mu e_3) +_\varphi h_{Z\hat{T}^2} (\lambda  e_1 + \mu e_3) \nonumber \\
& =h_{(\frac{\alpha \lambda}{1-\lambda}) ^{-1/n}Z((\frac{\alpha \lambda}{1-\lambda}) ^{1/n}T^{3})} (\lambda  e_1 + \mu e_3) +_\varphi h_{\alpha^{-1/n}Z(\alpha^{1/n}T^{3})} (e_1 +\lambda e_2 +\mu e_3)
\end{align*}
for any $0 < \mu < \lambda < 1$ and $\alpha >0$. Combining with (\ref{62}), the $SL(n)$ covariance of $Z$ and the homogeneity of support functions, we get
\begin{align}\label{62-5}
  a_2 +_\varphi \lambda h_{ZT^2} (\frac{\mu}{\lambda}  e_1 +  e_2)
=(\lambda a_2) +_\varphi h_{\alpha^{-1/n}Z(\alpha^{1/n}T^{3})} (e_1 +\lambda e_2 +\mu e_3)
\end{align}
for any $0 < \mu < \lambda < 1$ and $\alpha >0$. For fixed $\alpha$, let $\mu \to \lambda ^-$, by (\ref{47}) and the continuity of support functions, we get
\begin{align}\label{62-3}
  a_2 +_\varphi (\lambda a_1)
=(\lambda a_2) +_\varphi h_{\alpha^{-1/n}Z(\alpha^{1/n}T^{3})} (e_1 + \lambda e_2 + \lambda e_3).
\end{align}
Since the support function is sublinear, taking $\lambda = \frac{1}{2}$ in (\ref{62-3}), combining with (\ref{62}) and the $SL(n)$ covariance of $Z$, we have
\begin{align*}
  & a_2 +_\varphi (\frac{1}{2} a_1) \nonumber \\
& \leq (\frac{1}{2} a_2 ) +_\varphi \Big(h_{\alpha^{-1/n}Z(\alpha^{1/n}T^{3})} (\frac{1}{2} e_1 + \frac{1}{2} e_2 ) + h_{\alpha^{-1/n}Z(\alpha^{1/n}T^{3})} (\frac{1}{2} e_1 + \frac{1}{2} e_3 ) \Big) \nonumber \\
& = (\frac{1}{2} a_2 ) +_\varphi a_2.
\end{align*}

Note that $\varphi ^{-1} \{0 \} = [0,\eta]$, $0 \leq \eta <1$. If $\frac{a_1}{2a_2} > \eta$, by Proposition \ref{pro2.1} (\rmnum{5}), we have $$\frac{1}{2} a_1 \leq \frac{1}{2} a_2.$$
The proof is completed for this case.


If $\frac{a_1}{2a_2} \leq \eta$, taking $x=e_2$, $h_{ZT^2}(e_2) = a_2$, $h_{Z\hat{T}^1}(e_1) = h_{ZT^1}(e_1) = a_1$ in (\ref{55}), by the homogeneity of support functions, we get
\begin{align}\label{62-1}
  a_2 +_\varphi (1-\lambda) a_1
=h_{ZT^2} ((1-\lambda) e_1 + e_2) +_\varphi ((1-\lambda)a_2)
\end{align}
for any $0 < \lambda < 1$. Take $\frac{1}{2} \leq 1-\lambda = \eta \frac{a_2}{a_1} <1$ in (\ref{62-1}). Since $\frac{(1-\lambda)a_2}{a_2} \leq \frac{(1-\lambda)a_1}{a_2} = \eta$, by Proposition \ref{pro2.1} (\rmnum{6}), we get
\begin{align}\label{62-4}
h_{ZT^2} (\eta \frac{a_2}{a_1} e_1 + e_2) = a_2.
\end{align}
Then we infer from (\ref{62-4}), $\mu = \lambda \eta \frac{a_2}{a_1}$ in (\ref{62-5}), the homogeneity and the continuity of support functions and Proposition \ref{pro2.1} (\rmnum{2}) that
\begin{align}\label{62-7}
  h_{\alpha^{-1/n}Z(\alpha^{1/n}T^{3})} (e_1 +\lambda e_2 + \lambda \eta \frac{a_2}{a_1} e_3) =a_2
\end{align}
for any $0 \leq \lambda \leq 1$ and $\alpha >0$.

Choosing $\lambda$ such that $\eta \frac{a_2}{a_1} < \lambda \leq \frac{1}{2-\eta \frac{a_2}{a_1}}$ (which is possible since $\eta \frac{a_2}{a_1} < \frac{1}{2-\eta \frac{a_2}{a_1}}$ when $\eta \frac{a_2}{a_1} \neq 1$) in (\ref{62-3}), since the support function is sublinear, combining with (\ref{62}), (\ref{62-7}) and the $SL(n)$ covariance of $Z$, we have
\begin{align*}
  & a_2 +_\varphi (\lambda a_1) \\
& \leq (\lambda a_2 ) +_\varphi \left(h_{\alpha^{-1/n}Z(\alpha^{1/n}T^{3})} \left(\begin{array}{c} \lambda \\ \lambda \\ \lambda \eta \frac{a_2}{a_1} e_3 \end{array} \right)
+ h_{\alpha^{-1/n}Z(\alpha^{1/n}T^{3})} \left(\begin{array}{c} 1-\lambda \\ 0 \\ \lambda - \lambda \eta \frac{a_2}{a_1} \end{array} \right) \right) \\
& = (\lambda a_2 ) +_\varphi \left(\lambda h_{\alpha^{-1/n}Z(\alpha^{1/n}T^{3})} \left(\begin{array}{c} 1 \\ 1 \\ \eta \frac{a_2}{a_1} e_3 \end{array} \right)
+ (1-\lambda) h_{\alpha^{-1/n}Z(\alpha^{1/n}T^{3})} \left(\begin{array}{c} \frac{\lambda - \lambda \eta \frac{a_2}{a_1}}{1-\lambda} \\ 1 \\ 0 \end{array} \right) \right) \\
& = (\lambda a_2 ) +_\varphi a_2.
\end{align*}
Since $\frac{\lambda a_1}{a_2} > \eta$, by Proposition \ref{pro2.1} (\rmnum{5}), we have $$\lambda a_1 \leq \lambda a_2.$$
The proof is completed.
\end{proof}

Finally, we get the main results for the $SL(n)$ covariant case.

\begin{thm}\label{thm4.6}
Let $n \geq 3$. If $Z : \mathcal{P}_o ^n \to \langle \mathcal {K}_o ^n, +_\varphi \rangle$ is an $SL(n)$ covariant Orlicz valuation for $\varphi \in \Phi$ and $+_\varphi \neq +_p$ for any $p \geq 1$, then there exists a constant $a \geq 0$ such that
\begin{align}\label{f1}
ZP = aP
\end{align}
for all $P \in \mathcal{P}_o ^n$, or
\begin{align}\label{f2}
ZP = -aP
\end{align}
for all $P \in \mathcal{P}_o ^n$.
\end{thm}
\begin{proof}
By Lemma \ref{lem4.5}, either $h_{ZT^1}(e_1) = 0$ or $h_{ZT^1}(-e_1) = 0$. Assume (w.l.o.g.) that $h_{ZT^1}(-e_1) = 0$. Denoting $a := h_{ZT^1}(e_1)$, we will show that (\ref{f1}) holds true for all $P \in \mathcal{P}_o ^n$. The case $h_{ZT^1}(e_1) = 0$ is similar (and (\ref{f2}) holds true with $a := h_{ZT^1}(-e_1)$).

We first need to prove that (\ref{f1}) holds true for $sT^d$, where $s>0$ and $1 \leq d \leq n$. $Z \{ o \} = \{ o \}$ has been shown in (\ref{37}).

If $d=1$, By the $SL(n)$ covariance of $Z$, we have $Z[o, se_1] = sZ[o, e_1]$ for any $s > 0$. By Lemma \ref{lem4.1}, we get that $Z[o, e_1] = [o,ae_1]$. The case $d=1$ is done.

\vskip 5pt
Assume that the desired result holds true for dimension $d-1$, $2 \leq d \leq n$, we want to show that the desired result also holds true for dimension $d$.

Let $x \in \text{span} \{ e_1, \cdots,e_d\}$. We will show by induction on the number $m$ of coordinates of $x$ not equal to zero that
\begin{align}\label{126}
  h_{Z (sT^d)} (x) = h_{asT^d} (x).
\end{align}
That means $Z(sT^d)=asT^d$.

For $m=1$, (\ref{126}) holds true by (\ref{e1}), (\ref{e2}), the $SL(n)$ covariance of $Z$ and the homogeneity of the support function. Assume that $(\ref{126})$ holds true for $m-1$. We need to show that $(\ref{126})$ also holds true for $m$. By the $SL(n)$ covariance of $Z$, we can assume, w.l.o.g., that $x=x_1 e_1+\cdots,x_m e_m$, $x_1,\cdots,x_m \neq 0$.

Note that from (\ref{60}) and (\ref{55}), we have
\begin{align}\label{60a}
h_{Z(sT^{d})} (x) +_\varphi h_{Z(\lambda^{1/d}s\hat{T}^{d-1})} (\psi_1 ^t x)
=h_{Z(\lambda^{1/d} sT^{d})} (\psi_1 ^t x) +_\varphi h_{Z((1-\lambda)^{1/d}sT^{d})} (\psi_2 ^t x)
\end{align}
for $2 \leq d \leq n$, since $ZsT^d = sZT^d$ for any $s>0$ when $d < n$. Here $x=(x_1, \cdots, x_d)^t \in \mathbb{R}^d$, $\psi_1 ^t x = \lambda^{-1/d} (\lambda x_1 + (1-\lambda)x_2, x_2, x_3,\cdots,x_d)^t$ and $\psi_2 ^t x = (1-\lambda)^{-1/d}(x_1, \lambda x_1+(1-\lambda)x_2,x_3,\cdots,x_d)^t$.
We will use (\ref{60a}) to get $h_{ZT^d}$ for $2 \leq d \leq n$.

Let $x_1 > x_2 > 0$ or $0 > x_2 >x_1$. Taking $x=x_1 e_1 + x_3 e_3 + \cdots +x_m e_m$, $\lambda = \frac{x_2}{x_1}$, $s=(1-\lambda) ^{-1/d}s$ in (\ref{60a}), with (\ref{210}) and the homogeneity of Orlicz addition (\ref{9}), we get
\begin{align}\label{138}
& h_{(1-\lambda) ^{1/d}Z((1-\lambda) ^{-1/d}sT^d)} (x_1e_1 + x_3 e_3 + \cdots +x_m e_m) \nonumber \\
 & \qquad \qquad +_\varphi h_{(1-\lambda) ^{1/d}\lambda ^{-1/d}Z((1-\lambda) ^{-1/d}\lambda ^{1/d}s\hat{T}^{d-1}} (x_2 e_1 + x_3 e_3 + \cdots +x_m e_m)  \nonumber \\
& = h_{(1-\lambda) ^{1/d}\lambda ^{-1/d}Z((1-\lambda) ^{-1/d}\lambda ^{1/d}sT^d)} (x_2 e_1 + x_3 e_3 + \cdots +x_m e_m) \nonumber \\
 & \qquad \qquad +_\varphi h_{Z(s T^d)} (x_1 e_1 + x_2 e_2 +x_3 e_3 + \cdots +x_m e_m).
\end{align}

Since $|x_2| < |x_1|$, combining induction assumption (\ref{f1}) for $d-1$ and (\ref{126}) for $m-1$ with the $SL(n)$ covariance of $Z$, we have
\begin{align*}
&h_{(1-\lambda) ^{1/d}Z((1-\lambda) ^{-1/d}sT^d)} (x_1e_1 + x_3 e_3 + \cdots +x_m e_m)  \nonumber\\
& \qquad \qquad = \max \{ 0, as x_i : 1 \leq i \leq m,~\text{and}~i \neq 2\}  \nonumber\\
& \qquad \qquad \geq \max \{ 0, as x_i : 2 \leq i \leq m\}  \nonumber\\
& \qquad \qquad = h_{(1-\lambda) ^{1/d}\lambda ^{-1/d}Z((1-\lambda) ^{-1/d}\lambda ^{1/d}s\hat{T}^{d-1}} (x_2 e_1 + x_3 e_3 + \cdots +x_m e_m) \nonumber\\
& \qquad \qquad = h_{(1-\lambda) ^{1/d}\lambda ^{-1/d}Z((1-\lambda) ^{-1/d}\lambda ^{1/d}sT^d)} (x_2 e_1 + x_3 e_3 + \cdots +x_m e_m).
\end{align*}
Then we infer from (\ref{138}) and Proposition \ref{pro2.1} (\rmnum{2}) that
\begin{align}\label{149}
&h_{Z sT^d} (x_1 e_1 + x_2 e_2 +x_3 e_3 + \cdots +x_m e_m) \nonumber \\
& \qquad \qquad = h_{(1-\lambda) ^{1/d}Z(1-\lambda) ^{-1/d}sT^d} (x_1e_1 + x_3 e_3 + \cdots +x_m e_m) \nonumber \\
& \qquad \qquad = \max \{0, as x_i: 1 \leq i \leq m\}.
\end{align}

Let $x_2 > x_1 > 0$ or $0 > x_1 >x_2$.
Taking $x=x_2 e_2 + x_3 e_3 + \cdots +x_m e_m$, $1-\lambda = \frac{x_1}{x_2}$, $s=\lambda ^{-1/d}s$ in (\ref{60a}), with (\ref{210}) and the homogeneity of Orlicz addition (\ref{9}), we get
\begin{align}\label{140}
& h_{\lambda ^{1/d}Z(\lambda ^{-1/d}sT^d)} (x_2e_2 + x_3 e_3 + \cdots +x_m e_m) \nonumber \\
& \qquad \qquad +_\varphi h_{Zs\hat{T}^{d-1}} (x_1 e_1 +x_2 e_2 + x_3 e_3 + \cdots +x_m e_m) \nonumber \\
&= h_{Z(sT^d)} (x_1 e_1 +x_2 e_2 + x_3 e_3 + \cdots +x_m e_m) \nonumber \\
 & \qquad \qquad +_\varphi h_{Z(1-\lambda) ^{-1/d}\lambda ^{1/d}Z((1-\lambda) ^{1/d}\lambda ^{-1/d}s T^d)} (x_1 e_2 +x_3 e_3 + \cdots +x_m e_m).
\end{align}
Similarly to the case $|x_2| < |x_1|$, since $|x_2| > |x_1|$, we have
\begin{align*}
&h_{\lambda ^{1/d}Z(\lambda ^{-1/d}sT^d)} (x_2e_2 + x_3 e_3 + \cdots +x_m e_m) \\
& \qquad \geq h_{Zs\hat{T}^{d-1}} (x_1 e_1 +x_2 e_2 + x_3 e_3 + \cdots +x_m e_m) \\
& \qquad =h_{Z(1-\lambda) ^{-1/d}\lambda ^{1/d}Z((1-\lambda) ^{1/d}\lambda ^{-1/d}s T^d)} (x_1 e_2 +x_3 e_3 + \cdots +x_m e_m).
\end{align*}
By Proposition \ref{pro2.1} (\rmnum{2}), we get
\begin{align}\label{150}
&h_{Z sT^d} (x_1 e_1 + x_2 e_2 +x_3 e_3 + \cdots +x_m e_m) \nonumber \\
& \qquad \qquad = h_{\lambda ^{1/d}Z(\lambda ^{-1/d}sT^d)} (x_2e_2 + x_3 e_3 + \cdots +x_m e_m) \nonumber \\
& \qquad \qquad = \max \{0, as x_i : 1 \leq i \leq m\}.
\end{align}

Let $x_1 >0 > x_2$ or $x_2 > 0 > x_1$. Taking $0 < \lambda = \frac{x_2}{x_2 - x_1} <1$ and $x= x_1 e_1 + x_2 e_2 + x_3 e_3 + \cdots +x_m e_m$ in (\ref{60a}), we get
\begin{align}\label{123}
&h_{ZsT^d} (x_1 e_1 + x_2 e_2 + x_3 e_3 + \cdots +x_m e_m) \nonumber \\
& \qquad \qquad +_\varphi h_{\lambda ^{-1/d}Z(\lambda ^{1/d}s\hat{T}^{d-1})} (x_2 e_2 + x_3 e_3 + \cdots +x_m e_m) \nonumber \\
&= h_{\lambda ^{-1/d}Z(\lambda ^{1/d}sT^d)} (x_2 e_2 + x_3 e_3 + \cdots +x_m e_m) \nonumber \\
& \qquad \qquad +_\varphi h_{(1-\lambda) ^{-1/d}Z((1-\lambda) ^{1/d}sT^d)} (x_1 e_1 + x_3 e_3 + \cdots +x_m e_m).
\end{align}
Combining with the assumption (\ref{f1}) for $d-1$ and (\ref{126}) for $m-1$ and the $SL(n)$ covariance of $Z$, we have
\begin{align*}
&h_{\lambda ^{-1/d}Z(\lambda ^{1/d}s\hat{T}^{d-1})} (x_2 e_2 + x_3 e_3 + \cdots +x_m e_m) \\
&\qquad = h_{\lambda ^{-1/d}Z(\lambda ^{1/d}sT^d)} (x_2 e_2 + x_3 e_3 + \cdots +x_m e_m) \\
&\qquad \leq h_{(1-\lambda) ^{-1/d}Z((1-\lambda) ^{1/d}sT^d)} (x_1 e_1 + x_3 e_3 + \cdots +x_m e_m)
\end{align*}
for $x_1 >0 > x_2$ and
\begin{align*}
&h_{\lambda ^{-1/d}Z(\lambda ^{1/d}s\hat{T}^{d-1})} (x_2 e_2 + x_3 e_3 + \cdots +x_m e_m) \\
&\qquad = h_{(1-\lambda) ^{-1/d}Z((1-\lambda) ^{1/d}sT^d)} (x_1 e_1 + x_3 e_3 + \cdots +x_m e_m) \\
&\qquad \leq h_{\lambda ^{-1/d}Z(\lambda ^{1/d}sT^d)} (x_2 e_2 + x_3 e_3 + \cdots +x_m e_m)
\end{align*}
for $x_2 > 0 > x_1$.
In all, by Proposition \ref{pro2.1} (\rmnum{2}), we obtain that
\begin{align}\label{124}
h_{Z (sT^d)} (x_1 e_1 + x_2 e_2 +x_3 e_3 + \cdots +x_m e_m) = \max \{0, a sx_i : 1 \leq i \leq m \}.
\end{align}
Combining (\ref{149}), (\ref{150}), (\ref{124}) and the continuity of the support function, we get
\begin{align*}
h_{Z (sT^d)} (x_1e_1 + \cdots +x_m e_m) = h_{a sT^d} (x_1e_1 + \cdots +x_m e_m)
\end{align*}
for any $x_1,\cdots,x_m \in \mathbb{R}$.

\vskip 5pt

By the $SL(n)$ covariance of $Z$, (\ref{f1}) holds true for any simplex in $\mathcal {T}_o^n=\mathcal{P}_{1}$.
Assume that (\ref{f1}) holds on $\mathcal{P}_{i-1}$, $i \geq 2$. For $P=P_1 \cup P_2 \in \mathcal{P}_{i}$, where $P_1,P_2 \in \mathcal{P}_{i-1}$ have disjoint relative interiors, by (\ref{303}), we have $P_1 \cap P_2 \in \mathcal{P}_{i-1}$. Hence,
$$h_{Z(P_1 \cap P_2)} = h_{a(P_1 \cap P_2)} \leq h_{aP_i}= h_{ZP_i}$$
for $i=1,2$. Therefore $h_{Z(P_1 \cup P_2)}$ is uniquely determined by (\ref{16}) and (\ref{11a}), namely,
\begin{align*}
h_{Z(P_1 \cup P_2)}(x)
= (h_{ZP_1}(x) +_\varphi h_{ZP_2}(x)) \varphi ^{-1} \left(1 - \varphi \left( \frac{h_{Z(P_1 \cap P_2)}(x)}{h_{ZP_1}(x) +_\varphi h_{ZP_2}(x)}\right)\right),
\end{align*}
if $h_{ZP_1}(x)$ and $h_{ZP_2}(x)$ are not both equal to $0$; and $h_{Z(P_1 \cup P_2)}(x) =0$ if $h_{ZP_1}(x) = h_{ZP_2}(x)=0$. Here $x \in \mathbb{R}^n$.
Also since $Z$ defined by (\ref{f1}) is an Orlicz valuation, we get that (\ref{f1}) holds on $\mathcal{P}_i$. Hence, we conclude that (\ref{f1}) holds on $\mathcal{P}_i$ inductively for any $i$. For any $P \in \mathcal {P}_o^n$, there exists an $i$ such that $P \in \mathcal{P}_{i}$. Thus (\ref{f1}) holds for all $P \in \mathcal {P}_o^n$.
\end{proof}

\section*{Acknowledgement}
The authors would like to thank Prof. Monika Ludwig for \text{suggestions} to improve the original draft.
The work of the authors was \text{supported}, in part, by the \text{National} \text{Natural} Science Foundation of China (11271244) and Shanghai Leading Academic Discipline Project (S30104). The first author was also supported by China Scholarship Council (CSC 201406890044).


\end{document}